\documentclass[11pt]{article}

\usepackage{amssymb}
\usepackage{amsmath}
\usepackage{amsthm}
\usepackage{tikz-cd}

\numberwithin{equation}{section}

\newtheoremstyle{slplain}
  {\topsep}
  {\topsep}
  {\slshape}
  {0pt}
  {\bfseries}
  {.}
  {0.5em}
  {}

\theoremstyle{slplain}
  \newtheorem{THM}{Theorem}[section]
  \newtheorem{LEM}[THM]{Lemma}
  \newtheorem{PROP}[THM]{Proposition}
  \newtheorem{COR}[THM]{Corollary}

\theoremstyle{definition}
  \newtheorem{DEF}[THM]{Definition}

\renewcommand{\preceq}{\preccurlyeq}
\renewcommand{\succeq}{\succcurlyeq}
\newcommand\toCC[1]{\overset{#1}\longrightarrow}

\renewcommand{\le}{\leqslant}
\renewcommand{\ge}{\geqslant}

\newcommand{\0}{\varnothing}
\renewcommand{\phi}{\varphi}
\newcommand{\CC}{\mathbf{C}}
\newcommand{\DD}{\mathbf{D}}
\newcommand{\KK}{\mathbf{K}}
\newcommand{\NN}{\mathbb{N}}
\newcommand{\RR}{\mathbb{R}}
\renewcommand{\SS}{\mathbf{S}}
\newcommand{\union}{\cup}
\newcommand{\Boxed}[1]{\mbox{$#1$}}
\newcommand{\id}{\mathrm{id}}
\newcommand{\Ob}{\mathrm{Ob}}

\newcommand{\calA}{\mathcal{A}}
\newcommand{\calB}{\mathcal{B}}
\newcommand{\calC}{\mathcal{C}}
\newcommand{\calP}{\mathcal{P}}

\newcommand{\Cat}{\mathbf{Cat}}

\newcommand{\Ess}{\mathrm{Ess}}
\newcommand{\Aut}{\mathrm{Aut}}
\newcommand{\Part}{\mathrm{Part}}
\newcommand{\upset}[2]{\Boxed{\uparrow_{#1}}\mathstrut#2}
\newcommand{\subobj}[3]{\binom{#3}{#2}_{#1}}

\newcommand{\Bol}{{\mathit{Bol}}}
\newcommand{\Sha}{{\mathit{Sha}}}

\title{Ramsey degrees and entropy of combinatorial structures}
\author{%
  Dragan Ma\v sulovi\'c\\
  University of Novi Sad, Faculty of Sciences\\
  Department of Mathematics and Informatics\\
  Trg Dositeja Obradovi\'ca 3, 21000 Novi Sad, Serbia\\
  e-mail: dragan.masulovic@dmi.uns.ac.rs}

\begin{document}
\maketitle

\begin{abstract}
  Close connections between various notions of entropy and the apparatus of category theory
  have been observed already in the 1980s and more vigorously developed in the past ten years.
  The starting point of the paper is the recent categorical understanding of structural Ramsey degrees,
  which then leads to a way to compute entropy of an object in a small category not as a
  measure of statistical, but as a measure of its combinatorial complexity.
  The new entropy function we propose, the Ramsey entropy, is a real-valued invariant of an object
  in an arbitrary small category.
  We require no additional categorical machinery to introduce and prove the properties of this entropy.
  Motivated by combinatorial phenomena (structural Ramsey degrees) we build the necessary infrastructure
  and prove the fundamental properties using only special partitions imposed on homsets.
  We conclude the paper with the discussion of the maximal Ramsey entropy on
  a category that we refer to as the Ramsey-Boltzmann entropy.

  \bigskip

  \noindent \textbf{Key Words:} entropy, structural Ramsey degrees, finite structures

  \noindent \textbf{AMS Subj.\ Classification (2010):} 18A99, 05C55, 94A17
\end{abstract}

\section{Introduction}

Close connections between various notions of entropy and the apparatus of category theory
have been observed already in the 1980's~\cite{lawvere-cats-entropy-1984}
and more vigorously developed in the past ten years~\cite{baez-fritz-2014,baez-fritz-leinster-2011}.
In particular, \cite[Chapter 12]{leinster-entropy-diversity-2020} describes a general
categorical construction which specialized to a real line produces the Shannon entropy.
In this paper we add to the body of arguments in favor of the categorical treatment of entropy.
Our entropy function is a real invariant of an object in an arbitrary small category. We require no
additional categorical machinery to introduce and prove the properties of this entropy.
Motivated by combinatorial phenomena (structural Ramsey degrees) we build the necessary infrastructure
and prove the fundamental properties using only special partitions imposed on homsets.

We do not think of objects in a category as probability distributions and then introduce the
abstract entropy function fashioned after the Shannon entropy. Rather, our categories
are models of classes of finite first-order structures. Recent categorical understanding of
structural Ramsey degrees then leads to a way to compute entropy of an object in a small
category not as a measure of statistical, but as a measure of its combinatorial complexity.
It should, therefore, come as no surprise that it is the Boltzmann entropy, not the Shannon entropy,
that plays the key role in understanding the properties of the entropy function proposed by the paper.

Two central features of every entropy function are continuity and the additive,
or logarithmic property ($H(A, B) = H(A) + H(B)$). Due to the discrete nature
of finite structures (and, consequently, absence of any nontrivial intrinsic topology)
in this paper we adopt the approach advocated in \cite{lieb-yngvason-1999} where
beside the additive property entropy is required to be monotonous.
It is important to stress, however, that monotonicity is a form of semicontinuity
with respect to the topology induced by the (pre)ordering relation. 
Each class of finite structures is naturally preordered by embeddability: we let
$A \preceq B$ if there is an embedding $A \hookrightarrow B$.
So, the entropy functions we are interested in are those with
the additive property ($H(A, B) = H(A) + H(B)$)
which are monotonous ($A \preceq B \Rightarrow H(A) \le H(B)$).

The paper is organized as follows. In Section~\ref{rament.sec.prelim} we provide a brief overview of
notions related to partitions and various entropy functions on partitions,
we then recall some basic facts from category theory, and finally introduce several Ramsey-related phenomena
using the language of category theory.
In Section~\ref{rament.sec.rament} we introduce the notion of Ramsey entropy.
Our starting point is a small category whose morphisms are mono and homsets are finite.
In this setting we isolate special partitions imposed on homsets that we refer to as \emph{essential partitions}
and introduce the Ramsey entropy as a function related to an entropy function applied to such partitions.
We then use the close relationship between essential partitions and structural Ramsey degrees to
justify the name: we dare call this function entropy because it is subadditive and monotonous;
and we dare call it Ramsey because it is bounded by the logarithm of the structural Ramsey degree,
exists in case the category has finite Ramsey degrees and is zero on subobjects of Ramsey objects.

The rest of the paper is devoted to maximal Ramsey entropy referred to as the \emph{Ramsey-Boltzmann entropy.}
The behavior of the maximal entropy is governed by the properties of structural Ramsey degrees which
act as the corresponding diversity measure.
Most results in Section~\ref{rament.sec.r-b}, and in particular the additivity of the Ramsey-Boltzmann
entropy, then follow as immediate corollaries. The paper concludes with
the discussion of the behavior of the Ramsey-Boltzmann entropy under forgetful functors.
Since such a functor takes an ``object with more structure'' to an ``object
with less structure'', it makes sense to assume that the entropy should increase along the way.

\section{Motivation}

Thinking of entropy of a finite structure, for example a finite simple graph,
as the amount of information needed to describe the structure,
we expect this magnitude to be related to the size of the automorphism group of the structure:
structures with few automorphisms need complex descriptions and, thus,
should have high entropy, while structures with large automorphism groups are highly symmetric
and of low entropy. Whether the automorphism group of a structure is ``small'' or ``large''
is usually measured with respect to the size of the full symmetric group. So, the entropy
of a finite structure $A$ should be related to $\frac{n(A)!}{|\Aut(A)|}$, where $n(A)$ is the
size of~$A$. However, taking the entropy of $A$ to be simply the logarithm of this quantity
does not suffice for at least two reasons: in most cases $\log \frac{n(A)!}{|\Aut(A)|}$ is not
monotonous in the way we have indicated above;
on the other hand, while $\log \frac{n(A)!}{|\Aut(A)|}$ displays many desirable
properties in case of finite simple graphs, for other combinatorially interesting classes of
structures such as finite partial orders the additive property cannot be established. 
Interestingly, in case of finite partial orders
replacing $\frac{n(A)!}{|\Aut(A)|}$ with $\frac{e(A)}{|\Aut(A)|}$, where $e(A)$ is the number of
linear extensions of~$A$, yields a quantity possessing the same desirable properties.
This immediately suggests (see~\cite{fouche97,fouche98,fouche99} for computations of structural Ramsey degrees
for some prominent classes of finite structures) that the entropy function we are interested in should be related
to \emph{structural Ramsey degrees} of finite structures.

The most concise way to introduce various Ramsey phenomena is with the help of the \emph{Erd\H os-Rado partition arrow}.
The Erd\H os-Rado partition arrow was introduced in the 1950's as a convenient tool
to generalize Ramsey's theorem from the positive integers (that is, $\aleph_0$) to larger cardinals.
This successful generalization of Ramsey's theorem from sets (unstructured objects) to cardinals (special well-ordered chains)
prompted in the early 1970's the idea of generalizing the setup to arbitrary first-order structures, giving birth to
\emph{structural Ramsey theory}. 

For finite first-order structures $\calA, \calB, \calC$ over the same first-order language and for
a positive integer  $k \in \NN$, we write $\calC \longrightarrow (\calB)^\calA_{k}$ to denote that for every
coloring of the set $\binom \calB \calA$ of all the substructures of $\calB$
isomorphic to $\calA$ into $k$ colors (not all of which have to be used) there is a $\calB^* \in \binom \calC\calB$ such that
$\binom{\calB^*}{\calA}$ is \emph{monochromatic} (that is, all $\calA^* \in \binom{\calB^*}{\calA}$ are colored by the same color).
A class $\KK$ of finite first-order structures over the same first-order language
has the \emph{Ramsey property} if for every $\calA, \calB \in \KK$
there is a $\calC \in \KK$ such that for every $k \in \NN$ we have that $\calC \longrightarrow (\calB)^\calA_{k}$.

Alas, most combinatorially interesting classes of structures (finite graphs, finite partial orders etc)
\emph{do not} enjoy the Ramsey property. There are two ways to remedy this situation,
and, fortunately, they are equivalent. (The equivalence of the two approaches was observed and proved only a few years ago,
in~2016 in~\cite{Zucker-1}.)
Already in the 1970's it was observed that in combinatorially interesting cases adding carefully chosen linear orders yields
classes with the Ramsey property. For example, by adding all possible linear orders to finite graphs
we get a class of finite ordered graphs which is a Ramsey class; similarly, by adding all possible linear
extensions to finite partial orders we get another Ramsey class. On the other hand, if for some reason we
do not wish to modify the language, we can prove that almost all combinatorially
interesting classes of finite structures enjoy the weaker property of \emph{having
finite Ramsey degrees}. This phenomenon was observed already in the analysis of the Ramsey property for infinite
cardinals, and the first Ramsey degrees for combinatorially interesting classes of structures
were computed in~\cite{fouche97,fouche98,fouche99}.
An integer $t \ge 1$ is a \emph{Ramsey degree} of $\calA \in \KK$ if it is the smallest
positive integer satisfying the following: for any $k \in \NN$ and
any $\calB \in \KK$ there is a $\calC \in \KK$ such that
$$\displaystyle
  \calC \longrightarrow (\calB)^\calA_{k, t}.
$$
This is a symbolic way of expressing that 
no matter how we color the copies of $\calA$ in $\calC$ with $k$ colors, one can always find a \emph{$t$-oligochromatic} copy
$\calB^* \in \binom \calC\calB$ (that is, at most $t$ colors are used in the coloring of $\binom{\calB^*}\calA$).
If no such $t \ge 1$ exists for an $\calA \in \KK$, we say that $\calA$ \emph{does not have finite Ramsey degree.}
For example, finite graphs, finite partial orders and many other classes of finite structures
are known to have finite Ramsey degrees \cite{fouche97,fouche98,fouche99}.

As the structural Ramsey theory evolved, it has become evident that the Ramsey phenomena
depend not only on the choice of objects, but also on the choice of morphisms involved.
It was Leeb who pointed out already in the early 1970's \cite{leeb-cat,leeb-vorlesungen}
that the use of category theory can be quite helpful
both in the formulation and in proving results pertaining to structural Ramsey theory.
However, instead of pursuing the original approach by Leeb (which has very fruitfully been applied to a wide range of
Ramsey-type problems \cite{GLR, leeb-cat, Nesetril-Rodl}),
we proposed in~\cite{masulovic-ramsey} a systematic study of
a simpler approach motivated by and implicit in \cite{mu-pon,vanthe-more,Zucker-1}.

Another observation that crystallized over the years is the fact that we can and have to distinguish
between the Ramsey property for structures (where we color \emph{copies} of one structure within another structure)
and the Ramsey property for embeddings (where we color \emph{embeddings} of one structure into another structure).
Consequently, we shall introduce both structural and embedding Ramsey degree of an object.
Although structural Ramsey degrees in a category are true generalizations of Ramsey degrees for structures, it turns out that
embedding Ramsey degrees are easier to calculate with. Fortunately, the relationship between the two is straightforward,
as demonstrated in \cite{Zucker-1}, and it carries over to the abstract case of Ramsey degrees in
categories (see Proposition~\ref{rdbas.prop.sml}).

\section{Preliminaries}
\label{rament.sec.prelim}

\paragraph{Partitions.}
Let $X$ be a nonempty set. A \emph{partition of $X$} is a set $\Pi = \{\beta_i : i \in I\} \subseteq \calP(X)$
such that all the $\beta_i$'s are nonempty, $\bigcup_{i \in I} \beta_i = X$ and $i \ne j \Rightarrow \beta_i \cap \beta_j = \0$ for all $i, j \in I$.
A partition is finite if the index set $I$ is finite. The $\beta_i$'s will be referred to as the \emph{blocks} of~$\Pi$.
If $x, y \in X$ belong to the same block of $\Pi$ we shall write $x \equiv_\Pi y$.

By $\Part(X)$ we denote the set of all the partitions of $X$.
Recall that $\Part(X)$ can be partially ordered in a standard fashion:
for $\Sigma, \Pi \in \Part(X)$ we write $\Sigma \preceq \Pi$ when $\Pi$ is finer than $\Sigma$
(that is, for every $\beta \in \Pi$ there is a $\gamma \in \Sigma$ such that $\beta \subseteq \gamma$).
This ordering turns $\Part(X)$ into a lattice; in particular
every family $\Pi_i$, $i \in I$, of partitions of $X$ has the supremum $\bigvee_{i \in I} \Pi_i$
in the lattice $\Part(X)$. Note that the supremum refines partitions.

If $\Pi \in \Part(X)$ and $\Sigma \in \Part(Y)$ are two partitions not necessarily on the same set,
we say that $\Pi$ is \emph{isomorphic to} $\Sigma$, and write $\Pi \cong \Sigma$,
if there is a bijection $f : X \to Y$ such that $\Sigma = \{f(\beta) : \beta \in \Pi\}$.
We can also define the \emph{product} $\Pi \otimes \Sigma$ as follows:
$\Pi \otimes \Sigma = \{\beta \times \gamma : \beta \in \Pi, \gamma \in \Sigma\} \in \Part(X \times Y)$.
Note that $(x, y) \equiv_{\Pi \otimes\Sigma} (x', y')$ iff $x \equiv_\Pi x'$ and $y \equiv_\Sigma y'$.

We say that $H$ is an \emph{entropy on partitions} if $H$ is a family of functions
$H_X : \Part(X) \to \RR \union \{\infty\}$
indexed by finite sets $X$ such that the following holds for all finite sets $X$ and $Y$,
and all $\Pi, \Sigma \in \Part(X)$ and $\Lambda \in \Part(Y)$:
\begin{itemize}
\item $H_X(\Pi) \le \log |\Pi|$;
\item $H_X(\Pi) = 0$ if and only if $\Pi = \{X\}$ --- the trivial one-block partition;
\item if $\Sigma \preceq \Pi$ then $H_X(\Sigma) \le H_X(\Pi)$;
\item if $\Pi \cong \Lambda$ then $H_X(\Pi) = H_Y(\Lambda)$; and
\item $H_{X \times Y}(\Pi \otimes \Lambda) = H_X(\Pi) + H_Y(\Lambda)$.
\end{itemize}
The most prominent entropy function is definitely the \emph{Shannon entropy}:
$$
  H_X^\Sha(\Pi) = - \sum_{\beta \in \Pi} p(\beta) \log p(\beta)
$$
where $p(\beta) = |\beta| / |X|$. Nevertheless, in this paper we shall make use of the
\emph{Boltzmann entropy}:
$$
  H_X^\Bol(\Pi) = \log |\Pi|.
$$
Just to have things fixed, in this paper $\log x$ means $\log_2 x$.

\paragraph{Categories.}
Let us quickly fix some less standard notation. All categories in this paper are locally small.
The composition of morphisms in a category will be denoted by $\cdot$, e.g.\ $\id_B \cdot f = f = f \cdot \id_A$
for all $f \in \hom_\CC(A, B)$, and $(f \cdot g) \cdot h = f \cdot (g \cdot h)$ whenever the compositions are defined.
We write $A \toCC\CC B$ as a shorthand for $\hom_\CC(A, B) \ne \0$.
Note that $\toCC\CC$ is reflexive and transitive, so every category is naturally preordered.
If $\CC$ is a small category, for $A \in \Ob(\CC)$ let $\upset \CC A = \{B \in \Ob(\CC) : A \toCC\CC B\}$.
A category $\CC$ is \emph{directed} if for every $A, B \in \Ob(\CC)$ there is a $C \in \Ob(\CC)$
such that $A \toCC\CC C$ and $B \toCC\CC C$; and $\CC$ is a \emph{category with amalgamation}
if for every choice of objects $A, B, C \in \Ob(\CC)$ and morphisms $f \in \hom_\CC(A, B)$
and $g \in \hom_\CC(A, C)$ there is an object $D \in \Ob(\CC)$ and morphisms
$h \in \hom_\CC(B, D)$ and $k \in \hom_\CC(C, D)$ such that
$h \cdot f = k \cdot g$:
\begin{center}
  \begin{tikzcd}
    C \arrow[r, "k"] & D \\
    A \arrow[r, "f"'] \arrow[u, "g"] & B \arrow[u, "h"']
  \end{tikzcd}
\end{center}
Let $\Aut_\CC(A)$, the \emph{automorphism group} of $A \in \Ob(\CC)$,
be the set of all the invertible morphisms in $\hom_\CC(A, A)$.

A category $\DD$ is a \emph{cofinal subcategory} of a category $\CC$ if $\DD$ is a full subcategory of $\CC$
and for every $C \in \Ob(\CC)$ there is a $D \in \Ob(\DD)$ such that $C \toCC\CC D$.
The \emph{product} of categories $\CC$ and $\DD$ will be denoted by $\CC \times \DD$
and the iterated product $\underbrace{\CC \times \CC \times \ldots \times \CC}_n$
by $\CC^n$.

In all situations where the discussion obviously takes place in a
single category, usually indicated at the beginning of the paragraph or in the formulation of the statement,
instead of $A \toCC\CC B$, $\hom_\CC(A, B)$, $\upset{\CC} A$ \ldots\ we shall simply write $A \to B$, $\hom(A, B)$, $\upset{} A$ \ldots

\paragraph{Ramsey phenomena in a category.}
Let $\CC$ be a locally small category and let $A, B \in \Ob(\CC)$.
We think of $\hom(A, B)$ as the set of embeddings $A \to B$,
so we shall often require that all the morphisms in $\CC$ be mono.
Two morphisms $f, g \in \hom(A, B)$ ``point'' to the same subobject of $B$ isomorphic to $A$ if they
``differ by an automorphism''. More formally, write
$f \sim_A g$ to denote that there is an $\alpha \in \Aut(A)$ such that $f = g \cdot \alpha$.
It is easy to see that $\sim_A$ is an equivalence relation on $\hom(A, B)$, so we let
$
  \binom BA = \hom(A, B) / \Boxed{\sim_A}
$
be the set of all the \emph{subobjects of $B$ isomorphic to $A$}.

For a $k \in \NN$, a \emph{$k$-coloring} of a set $S$ is any mapping $\chi : S \to K$, where
$K$ is a finite set with $|K| = k$. Often it will be convenient to take
$k = \{0, 1,\ldots, k-1\}$ as the set of colors.
Each coloring $\chi : X \to k$ determines a partition $\Pi(\chi) = \{\chi^{-1}(i) : i < k\} \setminus \{\0\}$ of~$X$.
Conversely, each partition $\Pi = \{\beta_0, \beta_1, \ldots, \beta_{k-1}\} \in \Part(X)$ determines a coloring
$\chi_\Pi : X \to k$ by $\chi_\Pi(\beta_i) = i$. Note that $\chi_\Pi$ is determined only up to a permutation of blocks of~$\Pi$.

For an integer $k \in \NN$ and $A, B, C \in \Ob(\CC)$ we write
$
  C \overset\sim\longrightarrow (B)^{A}_{k, t}
$
to denote that for every $k$-coloring
$
  \chi : \subobj{} AC \to k
$
there is a morphism $w : B \to C$ such that $|\chi(w \cdot \subobj{} AB)| \le t$.
(Note that $w \cdot (f / \Boxed{\sim_A}) = (w \cdot f) / \Boxed{\sim_A}$ for $f / \Boxed{\sim_A} \in \subobj{}AB$;
therefore, we shall simply write $f \cdot g / \Boxed{\sim_A}$.)
Instead of $C \overset\sim\longrightarrow (B)^{A}_{k, 1}$ we simply write
$C \overset\sim\longrightarrow (B)^{A}_{k}$.

For $A \in \Ob(\CC)$ let $\tilde t_\CC(A)$ denote the least positive integer $n$ such that
for all $k \in \NN$ and all $B \in \Ob(\CC)$ there exists a $C \in \Ob(\CC)$ such that
$C \overset\sim\longrightarrow (B)^{A}_{k, n}$, if such an integer exists.
Otherwise put $\tilde t_\CC(A) = \infty$. Then $\tilde t_\CC(A)$ is referred to as the
\emph{structural Ramsey degree of $A$ in $\CC$}. 
A category $\CC$ \emph{has finite structural Ramsey degrees} if $\tilde t_\CC(A) < \infty$ for all $A \in \Ob(\CC)$.
An $A \in \Ob(\CC)$ is a \emph{Ramsey object in $\CC$} if $\tilde t_\CC(A) = 1$.
A $B \in \Ob(\CC)$ is a \emph{subramsey object in $\CC$} if $B \to A$ for some Ramsey object~$A$.
A locally small category $\CC$ has the
\emph{structural Ramsey property} if $\tilde t_\CC(A) = 1$ for all $A \in \Ob(\CC)$.

The notion of embedding Ramsey degrees can be introduced analogously and proves to be much more
convenient when it comes to actual calculations. We write
$
  C \longrightarrow (B)^{A}_{k,t}
$
to denote that for every $k$-coloring
$
  \chi : \hom(A, C) \to k
$
there is a morphism $w : B \to C$ such that $|\chi(w \cdot \hom(A, B))| \le t$.
Instead of $C \longrightarrow (B)^{A}_{k, 1}$ we simply write
$C \longrightarrow (B)^{A}_{k}$.

For $A \in \Ob(\CC)$ let $t_\CC(A)$ denote the least positive integer $n$ such that
for all $k \in \NN$ and all $B \in \Ob(\CC)$ there exists a $C \in \Ob(\CC)$ such that
$C \longrightarrow (B)^{A}_{k, n}$, if such an integer exists.
Otherwise put $t_\CC(A) = \infty$. Then $t_\CC(A)$ is referred to as the
\emph{embedding Ramsey degree of $A$ in $\CC$}. 
A locally small category $\CC$ has the
\emph{embedding Ramsey property} if $t_\CC(A) = 1$ for all $A \in \Ob(\CC)$.

The following relationship between structural and embedding Ramsey degrees
was proved for relational structures in \cite{Zucker-1} and generalized to this form in \cite{masul-kpt}.

\begin{PROP}\label{rdbas.prop.sml} (\cite{Zucker-1,masul-kpt})
  Let $\CC$ be a locally small category whose morphisms are mono
  and let $A \in \Ob(\CC)$. Then $t(A)$ is finite if and only if both $\tilde t(A)$ and $\Aut(A)$ are finite,
  and in that case
  $
    t(A) = |\Aut(A)| \cdot \tilde t(A)
  $.
\end{PROP}

\paragraph{Convention.}
Let $\NN_\infty = \NN \union \{\infty\} = \{1, 2, 3, \ldots, \infty \}$.
The usual linear order on the positive integers extends to $\NN_\infty$ straightforwardly:
$
  1 < 2 < \ldots < \infty
$.
Ramsey degrees take their values in $\NN_\infty$, so when we write
$t_1 \ge t_2$ for some Ramsey degrees $t_1$ and $t_2$ then
  $t_1, t_2 \in \NN$ and $t_1 \ge t_2$; or
  $t_1 = \infty$ and $t_2 \in \NN$; or
  $t_1 = t_2 = \infty$.
For notational convenience, if $A$ is an infinite set we shall simply write
$|A| = \infty$ regardless of the actual cardinal~$|A|$. Hence, if $t$ is a Ramsey degree
and $A$ is a set, by $t \ge |A|$ we mean the following:
  $t \in \NN$, $|A| \in \NN$ and $t \ge |A|$; or
  $t = \infty$ and $|A| \in \NN$; or
  $A$ is an infinite set and $t = \infty$.
On the other hand, if $A$ and $B$ are sets then $|A| \ge |B|$ has the usual meaning.

With this convention in mind Proposition~\ref{rdbas.prop.sml} takes the following much simpler form:
$t(A) = |\Aut(A)| \cdot \tilde t(A)$ for all $A \in \Ob(\CC)$.

The following is a technical result from~\cite{masul-rpppg} which will
form the basis for the key properties of the Ramsey entropy.

\begin{THM}\label{rpppg.thm.srd-mult} \cite{masul-rpppg}
  Let $\CC_1$ and $\CC_2$ be categories whose morphisms are mono and homsets are finite. Then,
  with Convention in mind, the following holds for all $A_1 \in \Ob(\CC_1)$ and $A_2 \in \Ob(\CC_2)$:
  \begin{itemize}
    \item $t_{\CC_1 \times \CC_2}(A_1, A_2) = t_{\CC_1}(A_1) \cdot t_{\CC_2}(A_2)$, and
    \item $\tilde t_{\CC_1 \times \CC_2}(A_1, A_2) = \tilde t_{\CC_1}(A_1) \cdot \tilde t_{\CC_2}(A_2)$.
  \end{itemize}
\end{THM}

We focus next on the monotonicity of Ramsey degrees. The main result we shall need,
Lemma~\ref{rament.lem.t-monotonous}, was first proved in \cite{Zucker-1} for amalgamation classes of structures,
and reproved in a rather convoluted manner in \cite{masul-kpt} for locally small directed categories
with amalgamation. Here, we give a short direct proof of a slightly more general statement
(we no longer require the category to be directed). We start with a technical statement.

\begin{LEM}\label{rament.lem.APext}
  Let $\CC$ be a category with amalgamation whose homsets are finite.
  For all $A, B, C \in \Ob(\CC)$ and every morphism $f \in \hom(A, B)$ there is a $D \in \Ob(\CC)$
  and a morphism $w \in \hom(C, D)$ such that $w \cdot \hom(A, C) \subseteq \hom(B, D) \cdot f$.
\end{LEM}
\begin{proof}
  Let $\hom(A, C) = \{g_1, g_2, \ldots, g_n\}$. By the amalgamation property, for every $i \in \{1, \ldots, n\}$
  there is a $D_i \in \Ob(\CC)$ and morphisms $p_i \in \hom(C, D_i)$ and $q_i \in \hom(B, D_i)$ such that
  \begin{center}
    \begin{tikzcd}
      B \arrow[r, "q_i"] & D_i \\
      A \arrow[u, "f"] \arrow[r, "g_i"'] & C \arrow[u, "p_i"']
    \end{tikzcd}
  \end{center}
  Moreover, by the repeated use of the amalgamation property there is a $D \in \Ob(\CC)$ and morphisms
  $s_i \in \hom(D_i, D)$, $1 \le i \le n$, such that
  \begin{center}
    \begin{tikzcd}
         &    & D & \\
      D_1 \arrow[urr, "s_1"] & D_2 \arrow[ur, "s_2"'] & {\cdots} & D_n \arrow[ul, "s_n"'] \\
         &    & C \arrow[ull, "p_1"] \arrow[ul, "p_2"'] \arrow[ur, "p_n"']  & \\
    \end{tikzcd}
  \end{center}
  Let $w = s_1 \cdot p_1 = \ldots = s_n \cdot p_n \in \hom(C, D)$. To see that $D$ and $w$ satisfy the requirements of the
  lemma take any $g_i \in \hom(A, C)$ and note that
  $$
    w \cdot g_i = s_i \cdot p_i \cdot g_i = s_i \cdot q_i \cdot f \in \hom(B, D) \cdot f.
  $$
  This concludes the proof.
\end{proof}

\begin{LEM}\label{rament.lem.t-monotonous} (cf.~\cite{Zucker-1,masul-kpt})
  Let $\CC$ be a category with amalgamation whose morphisms are mono and whose homsets are finite.
  Then for all $A_1, A_2 \in \Ob(\CC)$, if $A_1 \to A_2$ then $t(A_1) \le t(A_2)$.
\end{LEM}
\begin{proof}
  Let $t(A_2) = t \in \NN$ (the case $t(A_2) = \infty$ is trivial) and let us show that $t(A_2) \le t$.
  Take any $k \in \NN$ and $B_1 \in \Ob(\CC)$ so that $A_1 \toCC{} B_1$.
  Since $A_1 \toCC A_2$ choose and fix an $f \in \hom(A_1, A_2)$.
  By Lemma~\ref{rament.lem.APext} there is a $B_2 \in \Ob(\CC)$ and a $w \in \hom(B_1, B_2)$
  such that
  \begin{equation}\label{rament.eq.x-21}
    w \cdot \hom(A_1, B_1) \subseteq \hom(A_2, B_2) \cdot f.
  \end{equation}
  Since $t(A_2) = t$, for $B_2$ we have just constructed there is a $C \in \Ob(\CC)$ such that
  \begin{align}
    (\forall \chi_2 : \hom(A_2, C) \to k) (\exists w_2 &\in \hom(B_2, C))\nonumber\\
             &|\chi_2(w_2 \cdot \hom(A_2, B_2))| \le t. \label{rament.eq.x-22}
  \end{align}
  Now, take any $\chi_1 : \hom(A_1, C) \to k$ and define
  $\chi_2 : \hom(A_2, C) \to k$ by
  $$
    \chi_2(h) = \chi_1(h \cdot f).
  $$
  By \eqref{rament.eq.x-22} there is a $w_2 \in \hom(B_2, C)$ such that
  $$
    |\chi_2(w_2 \cdot \hom(A_2, B_2))| \le t.
  $$
  Finally, let $w' = w_2 \cdot w$. Then:
  \begin{align*}
    |\chi_1(w' \cdot \hom(A_1, B_1))|
    & = |\chi_1(w_2 \cdot w \cdot \hom(A_1, B_1))|\\
    & \le |\chi_1(w_2 \cdot \hom(A_2, B_2) \cdot f)| && \text{[by \eqref{rament.eq.x-21}]}\\
    & = |\chi_2(w_2 \cdot \hom(A_2, B_2))| && \text{[definition of $\chi_2$]}\\
    & \le t.
  \end{align*}
  This concludes the proof.
\end{proof}

\section{Ramsey entropy}
\label{rament.sec.rament}

In this section we introduce the notion of Ramsey entropy.
We develop the basic notions in a very general setting -- that of small categories whose morphisms
are mono and homsets are finite. Our notion of entropy is based on the fact that Ramsey degrees
are closely related to special partitions of subobjects.
We start by introducing the necessary infrastructure, and then prove that the notion we have introduced
possesses some fundamental properties typical for nonadditive entropies: it is monotonous and subadditive.

Fix a locally small category $\CC$ whose morphisms are mono.
Let $A, B, C \in \Ob(\CC)$ be chosen so that $A \to B \to C$ and let $w \in \hom(B, C)$.
Let
$$
  \ell_w : \hom(A, B) \to \hom(A, C) : f \mapsto w \cdot f
$$
denote the left multiplication by $w$.
Then for $F \subseteq \hom(A, C)$, $\beta \subseteq \subobj{} A C$ and $\Pi \in \Part\subobj{} A C$ let
\begin{align*}
  \ell_w^{-1}(F) &= \{f \in \hom(A, B) : \ell_w(f) \in F\}, \\
  \ell_w^{-1}(\beta) &= \textstyle \{\ell_w^{-1}(F) : F \in \beta \} \setminus \{\0\}, \\
  \ell_w^{-1}(\Pi) &= \{\ell_w^{-1}(\beta) : \beta \in \Pi\} \setminus \{\0\}.
\end{align*}
The following lemma ensures that $\ell_w^{-1}(\Pi) \in \Part\subobj{} AB$ whenever $\Pi \in \Part\subobj{} AC$
and $w \in \hom(B, C)$. This will, then, ensure that the main definition of the paper is correct.

\begin{LEM}\label{rament.lem.basic-prop}
  Let $\CC$ be a locally small category whose morphisms are mono.
  Let $A, B, C \in \Ob(\CC)$ be chosen so that $A \to B \to C$ and let $w \in \hom(B, C)$.
  
  $(a)$ For every $g \in \hom(A, B)$, $\ell_w^{-1}(w \cdot g / \Boxed{\sim_A}) = g / \Boxed{\sim_A}$.

  $(b)$ For every $f \in \hom(A, C)$ either $\ell_w^{-1}(f / \Boxed{\sim_A}) = \0$ or
  there is a $g \in \hom(A, B)$ such that $\ell_w^{-1}(f / \Boxed{\sim_A}) = g / \Boxed{\sim_A}$.

  $(c)$ For every $\Pi \in \Part\subobj{} A C$ we have that $\ell_w^{-1}(\Pi) \in \Part\subobj{} A B$.

  $(d)$ For every $k \in \NN$, $\chi : \subobj{}AC \to k$ and $\Lambda \in \Part\subobj{} AB$:
        $\Lambda \succeq \ell_w^{-1}(\Pi(\chi))$ if and only if
        $f / \Boxed{\sim_A} \equiv_\Lambda g / \Boxed{\sim_A} \Rightarrow \chi(w \cdot f / \Boxed{\sim_A}) =
        \chi(w \cdot g / \Boxed{\sim_A})$.
\end{LEM}
\begin{proof}
  $(a)$ Let us only prove inclusion $(\subseteq)$. Take any $x \in \ell_w^{-1}(w \cdot g / \Boxed{\sim_A})$.
  Then $w \cdot x \in w \cdot g / \Boxed{\sim_A}$, so there is an $\alpha \in \Aut(A)$ such that
  $w \cdot x = w \cdot g \cdot \alpha$. Since $w$ is mono if follows that $x = g \cdot \alpha$ whence
  $x \in g / \Boxed{\sim_A}$.
  
  $(b)$ Take an $f \in \hom(A, C)$ and assume that $\ell_w^{-1}(f / \Boxed{\sim_A}) \ne \0$. Take any
  $g \in \ell_w^{-1}(f / \Boxed{\sim_A})$ and let us show that $\ell_w^{-1}(f / \Boxed{\sim_A}) = g / \Boxed{\sim_A}$.
  Again, we shall only prove inclusion $(\subseteq)$. Take any $x \in \ell_w^{-1}(f / \Boxed{\sim_A})$. Then
  $w \cdot x \in f / \Boxed{\sim_A}$, so $w \cdot x \cdot \alpha = f$ for some $\alpha \in \Aut(A)$.
  By the same argument, $g \in \ell_w^{-1}(f / \Boxed{\sim_A})$ yields that
  $w \cdot g \cdot \alpha' = f$ for some $\alpha' \in \Aut(A)$. Therefore,
  $w \cdot x \cdot \alpha = f = w \cdot g \cdot \alpha'$. Since $w$ is mono,
  $x \cdot \alpha = g \cdot \alpha'$, so $x = g \cdot \alpha''$ for $\alpha'' = \alpha' \cdot \alpha^{-1} \in \Aut(A)$.
  This proves that $x \in g / \Boxed{\sim_A}$.
  
  $(c)$ Take any $\Pi \in \Part\subobj{} A C$ and let $\Pi = \{\beta_i : i \in I\}$ where $i \ne j \Rightarrow \beta_i \cap \beta_j = \0$.
  Then $\ell_w^{-1}(\Pi) = \{\ell_w^{-1}(\beta_i) : i \in I\} \setminus \{\0\}$.
  Note that $\bigcup_{i \in I} \ell_w^{-1}(\beta_i) = \subobj{} A B$ follows from $(a)$ and $(b)$.
  Finally, take any $i, j \in I$ such that $i \ne j$ and $\ell_w^{-1}(\beta_i) \ne \0 \ne \ell_w^{-1}(\beta_j)$,
  and let us show that $\ell_w^{-1}(\beta_i) \cap \ell_w^{-1}(\beta_j) = \0$.
  Suppose this is not the case. Then there exist $f_i, f_j \in \hom(A, C)$ such that $f_i / \Boxed{\sim_A} \in \beta_i$,
  $f_j / \Boxed{\sim_A} \in \beta_j$ and $\ell_w^{-1}(f_i / \Boxed{\sim_A}) = \ell_w^{-1}(f_j / \Boxed{\sim_A})$.
  By $(b)$ there exist $g_i, g_j \in \hom(A, B)$ such that $\ell_w^{-1}(f_i / \Boxed{\sim_A}) = g_i / \Boxed{\sim_A}$ and
  $\ell_w^{-1}(f_j / \Boxed{\sim_A}) = g_j / \Boxed{\sim_A}$. Since $g_i / \Boxed{\sim_A} = g_j / \Boxed{\sim_A}$
  there is an $\alpha \in \Aut(A)$ such that $g_i \cdot \alpha = g_j$. Moreover, there exist $\alpha_i, \alpha_j \in \Aut(A)$ such
  that $w \cdot g_i \cdot \alpha_i = f_i$ and $w \cdot g_j \cdot \alpha_j = f_j$, so
  $w \cdot g_i \cdot \alpha \cdot \alpha_j = f_j$. Therefore, $w \cdot g_i \in f_i / \Boxed{\sim_A}$ and
  $w \cdot g_i \in f_j / \Boxed{\sim_A}$. Hence, $f_i / \Boxed{\sim_A} = f_j / \Boxed{\sim_A}$, so $\beta_i \cap \beta_j \ne \0$.
  Contradiction.
  
  $(d)$ Let $\Sigma = \ell_w^{-1}(\Pi(\chi))$ and note that
  $f / \Boxed{\sim_A} \equiv_\Sigma g / \Boxed{\sim_A}$ if and only if $\chi(w \cdot f / \Boxed{\sim_A}) = \chi(w \cdot g / \Boxed{\sim_A})$.
\end{proof}

Let $\CC$ be a locally small category whose morphisms are mono. Fix $A, B \in \Ob(\CC)$ so that $A \to B$.
A partition $\Lambda \in \Part\subobj{} A B$ is \emph{essential} if
there is a $C \in \Ob(\CC)$ such that $B \to C$ and
for every partition $\Pi \in \Part \subobj {} A C$ there is a $w \in \hom(B, C)$
such that $\Lambda \succeq \ell_w^{-1}(\Pi)$.
Let $\Ess\subobj{} AB$ be the set of all the essential partitions of $\subobj {} A B$.
Where necessary we shall write $\Ess\subobj\CC AB$ to indicate in which category the essential partitions
are considered.

\begin{LEM}\label{rament.lem.Ess-nonempty}
  Let $\CC$ be a category whose morphisms are mono and homsets are finite. Let $A, B, C \in \Ob(\CC)$
  be chosen so that $A \to B \to C$ and let $\Part\subobj{}AC = \{\Pi_1, \Pi_2, \ldots, \Pi_n\}$.
  Further, let $w_1, w_2, \ldots, w_n \in \hom(B, C)$ be arbitrary.
  Then $\bigvee_{i = 1}^n \ell_{w_i}^{-1}(\Pi_i) \in \Ess\subobj{}AB$.
  In particular, $\Ess\subobj{}AB \ne \0$ for all $A, B \in \Ob(\CC)$ such that $A \to B$.
\end{LEM}
\begin{proof}
  To show that $\bigvee_{i = 1}^n \ell_{w_i}^{-1}(\Pi_i)$ is essential
  take any $\Pi_{i_0} \in \Part\subobj{}AC$. Then trivially
  $\bigvee_{i = 1}^n \ell_{w_i}^{-1}(\Pi_i) \succeq \ell_{w_{i_0}}^{-1}(\Pi_{i_0})$.
\end{proof}

\begin{DEF}
  Let $\CC$ be a small category whose morphisms are mono and homsets are finite, and let $H$ be an entropy on partitions.
  Define $\tilde r : \Ob(\CC) \to \RR \cup \{\infty\}$ as follows:
  $$
    \tilde r(X) = \inf\nolimits_{A \in \upset{}X} \sup\nolimits_{B \in \upset{}A} \min\nolimits_{\Lambda \in \Ess\subobj{}AB} H(\Lambda).
  $$
  We refer to $\tilde r$ as the \emph{Ramsey entropy based on $H$}, and say that \emph{$\CC$ admits the Ramsey entropy
  based on $H$} if $\tilde r(X) < \infty$ for all $X \in \Ob(\CC)$. In particular, if $\tilde r$ is based on the Boltzmann entropy
  $H^\Bol$, we refer to $\tilde r$ as the \emph{Ramsey-Boltzmann entropy}. Where necessary, we shall write $\tilde r_\CC$.
\end{DEF}

The following is straightforward, so we omit the proof.

\begin{LEM}\label{rament.lem.inv-iso}
  Let $\CC$ be a small category whose morphisms are mono and homsets are finite.
  If $X_1 \cong X_2 $ then $\tilde r(X_1) = \tilde r(X_2)$, for all $X_1, X_2 \in \Ob(\CC)$.
  \qed
\end{LEM}

The following lemma shows that essential partitions are closely related to structural Ramsey degrees
providing, thus, the link between essential partitions, structural Ramsey degrees and the Ramsey entropy.

\begin{LEM}\label{rament.lem.ess-col}
  Let $\CC$ be a category whose morphisms are mono and homsets are finite.
  
  $(a)$ For a $t \in \NN$ and an $A \in \Ob(\CC)$,
  $\tilde t(A) \le t$ if and only if
  for all $B \in \Ob(\CC)$ such that $A \to B$
  there is a partition $\Lambda \in \Ess\subobj{}AB$ such that $|\Lambda| \le t$.
  
  $(b)$ For all $A \in \Ob(\CC)$, $\tilde t(A) = \sup_{B : A \to B} \min_{\Lambda \in \Ess\subobj{}AB} |\Lambda|$.
\end{LEM}
\begin{proof}
  $(a)$ Fix an $A \in \Ob(\CC)$ and a $t \in \NN$.
  
  $(\Leftarrow)$: Take any $B \in \Ob(\CC)$ such that $A \to B$.
  By the assumption there is a partition $\Lambda \in \Ess\subobj{}AB$ such that $|\Lambda| \le t$.
  Therefore, there is a $C \in \Ob(\CC)$ such that $B \to C$ and for every partition $\Pi \in \Part\subobj{}AC$
  there is a $w \in \hom(B, C)$ such that $\Lambda \succeq \ell_w^{-1}(\Pi)$. Take any 
  $k \in \NN$ and any coloring $\chi : \subobj{}AC \to k$. By the choice of $\Lambda$ we have that
  $\Lambda \succeq \ell_w^{-1}(\Pi(\chi))$. Then
  $$\textstyle
    |\chi(w \cdot \subobj{}AB)| = |\ell_w^{-1}(\Pi(\chi))| \le |\Lambda| \le t.
  $$
  (To justify the equality let $\beta_i = \chi^{-1}(i)$; then $\ell_w^{-1}(\beta_i) \ne \0$ iff $i \in \chi(w \cdot \subobj{}AB)$.)

  $(\Rightarrow)$: Aiming for a contradiction,
  suppose that $\tilde t(A) \le t$, but that there is a $B \in \Ob(\CC)$ with $A \to B$ such that no
  partition $\Lambda \in \Part\subobj{}AB$ with at most $t$ blocks is essential.
  Let $\{\Lambda_0, \ldots, \Lambda_{n-1}\}$ be the set of all the partitions of $\subobj{}AB$ with at most $t$ blocks.
  
  From $\tilde t(A) \le t$ it then follows that for this particular $B$
  there is a $C$ such that for every $k \in \NN$ and every $\chi : \subobj{}AC \to k$ there is a $w \in \hom(B, C)$
  satisfying
  $
    |\chi(w \cdot \subobj{}AB)| \le t
  $.

  On the other hand, since no $\Lambda_i$ is essential, $i < n$, for every such $\Lambda_i$
  there exist a $\Sigma_i \in \Part\subobj{}AC$ such that
  \begin{equation*}
    (\forall w \in \hom(B, C)) \; \Lambda_i \not\succeq \ell_w^{-1}(\Sigma_i).
  \end{equation*}
  For the sake of convenience, let $\chi_i = \chi_{\Sigma_i} : \Part\subobj{}AC \to k_i$ be a coloring induced by $\Sigma_i$
  and let $\equiv_i$ be an abbreviation of $\equiv_{\Lambda_i}$. Then the last fact is equivalent to
  \begin{equation}\label{rament.eq.IFF}
    \begin{array}{r@{}l}
    (\forall \; &w \in \hom(B, C))(\exists f / \Boxed{\sim_A}, g / \Boxed{\sim_A} \in \subobj{}AB)\\
                &(f / \Boxed{\sim_A} \equiv_i g / \Boxed{\sim_A} \land \chi_i(w \cdot f / \Boxed{\sim_A}) \ne \chi_i(w \cdot g / \Boxed{\sim_A})).
    \end{array}
  \end{equation}
  Consider the coloring
  $$\textstyle
    \chi^* : \subobj{}AC \to k_0 \times k_1 \times \ldots \times k_{n-1}
  $$
  given by
  $$
    \chi^*(h / \Boxed{\sim_A}) = (\chi_{0}(h / \Boxed{\sim_A}), \chi_{1}(h / \Boxed{\sim_A}), \ldots, \chi_{n-1}(h / \Boxed{\sim_A})).
  $$
  Since $\tilde t(A) \le t$ there is a $w \in \hom(B, C)$ such that $|\chi^*(w \cdot \subobj{}AB)| \le t$.
  Let $s = |\chi^*(w \cdot \subobj{}AB)|$ and put
  $$
    \Gamma = \ell_w^{-1}(\Pi(\chi^*)).
  $$
  Clearly, $\Gamma$ is a partition of $\Part\subobj{}AB$ with $s \le t$ blocks,
  so $\Gamma = \Lambda_i$ for some $i < n$.
  By \eqref{rament.eq.IFF} there exist $f / \Boxed{\sim_A}, g / \Boxed{\sim_A} \in \subobj{}AB$ such that
  $$
    f / \Boxed{\sim_A} \equiv_{i} g / \Boxed{\sim_A} \text{ and } \chi_{i}(w \cdot f / \Boxed{\sim_A}) \ne \chi_{i}(w \cdot g / \Boxed{\sim_A}).
  $$
  Note that $f / \Boxed{\sim_A} \equiv_{i} g / \Boxed{\sim_A}$ actually means $f / \Boxed{\sim_A} \equiv_{\Gamma} g / \Boxed{\sim_A}$ because $\Gamma = \Lambda_i$,
  whence $\chi^*(w \cdot f / \Boxed{\sim_A}) = \chi^*(w \cdot g / \Boxed{\sim_A})$ (because $\Gamma = \ell_w^{-1}(\Pi(\chi^*))$).
  But projecting the latter onto the $i$th coordinate gives
  $\chi_{i}(w \cdot f / \Boxed{\sim_A}) = \chi_{i}(w \cdot g / \Boxed{\sim_A})$. Contradiction.

  $(b)$
  Directly from $(a)$.
\end{proof}

For a small category $\CC$ it will be convenient for future calculations
to introduce the following function $\Ob(\CC) \to \RR \union \{\infty\}$:
$$
  \phi(A) = \sup\nolimits_{B \in \upset{}A} \min\nolimits_{\Lambda \in \Ess\subobj{}AB} H(\Lambda)
$$
so that $\tilde r(X) = \inf_{A \in \upset{}X} \phi(A)$. Where necessary, we shall write $\phi_\CC$.

\begin{LEM}\label{rament.lem.phi<=log.srd}
  Let $\CC$ be a small category whose morphisms are mono and homsets are finite.
  Then $\phi(A) \le \log \tilde t(A)$ for all $A \in \Ob(\CC)$, where we assume $\log \infty = \infty$.
\end{LEM}
\begin{proof}
  Take any $A \in \Ob(\CC)$. If $\tilde t(A) = \infty$ the statement is trivial. Assume, therefore, that
  $\tilde t(A) = t \in \NN$. By Lemma~\ref{rament.lem.ess-col}~$(a)$, for every $B \in \Ob(\CC)$ such
  that $A \to B$ there is a partition $\Lambda \in \Ess\subobj{}AB$ with $|\Lambda| \le t$.
  Now, fix a $B_0 \in \Ob(\CC)$ and let $\Lambda_0 \in \Ess\subobj{}A{B_0}$ be a partition with $|\Lambda_0| \le t$. Then
  $$
    \min\nolimits_{\Lambda \in \Ess\subobj{}AB} \; H(\Lambda) \le H(\Lambda_0) \le \log t,
  $$
  whence we immediately get that
  $$
    \phi(A) = \sup\nolimits_{B \in \upset{}A} \min\nolimits_{\Lambda \in \Ess\subobj{}AB} H(\Lambda) \le \log t.
  $$
  This concludes the proof.
\end{proof}

\begin{THM}\label{rament.thm.R<=log-t}
  Let $\CC$ be a small category whose morphisms are mono and homsets are finite.
  Let $H$ be an arbitrary entropy on partitions and let $\tilde r$ be the Ramsey entropy based on~$H$.

  $(a)$ If $X \to Y$ for some $X, Y \in \Ob(\CC)$ then $\tilde r(X) \le \tilde r(Y)$;

  $(b)$ $\tilde r(X) \le \log \tilde t(X)$ for all $X \in \Ob(\CC)$ (where we assume $\log \infty = \infty$).

  $(c)$ If $\CC$ has finite structural Ramsey degrees then $\CC$ admits the Ramsey entropy.

  $(d)$ If $X \in \Ob(\CC)$ is a subramsey object in $\CC$ then $\tilde r(X) = 0$; in particular if $X$ is a Ramsey object in $\CC$
  then $\tilde r(X) = 0$.
\end{THM}
\begin{proof}
  $(a)$ Immediate from the definition of $\tilde r$.
  
  $(b)$ This follows immediately from Lemma~\ref{rament.lem.phi<=log.srd}:
  $\tilde r(X) \le \phi(X) \le \log \tilde t(X)$.

  $(c)$ This is an immediate consequence of $(b)$.

  $(d)$ Let $X \in \Ob(\CC)$ be a subramsey object. Then some $A \in \upset{}X$ is a Ramsey object,
  so $\tilde r(X) \le \phi(A) \le \log \tilde t(A) = 0$.
\end{proof}

We conclude this section by showing that Ramsey entropy based on arbitrary entropy on partitions is subadditive.
We start by considering products of essential partitions.

\begin{LEM}\label{rament.lem.L*L-in-Ess}
  Let $\CC_1$ and $\CC_2$ be categories whose morphisms are mono and homsets are finite.
  Let $A_1, B_1 \in \Ob(\CC_1)$ and $A_2, B_2 \in \Ob(\CC_2)$ be chosen so that $A_1 \to B_1$ and $A_2 \to B_2$.
  Then
  $$\textstyle
    \Ess\subobj{}{A_1}{B_1} \otimes \Ess\subobj{}{A_2}{B_2} \subseteq \Ess\subobj{}{(A_1, A_2)}{(B_1, B_2)},
  $$
  where $\Ess\subobj{}{A_1}{B_1} \otimes \Ess\subobj{}{A_2}{B_2} = \{\Lambda_1 \otimes \Lambda_2 : \Lambda_1 \in \Ess\subobj{}{A_1}{B_1}
  \text{ and } \Lambda_2 \in \Ess\subobj{}{A_2}{B_2}\}$. (Note that $(A_1, A_2), (B_1, B_2) \in \Ob(\CC_1 \times \CC_2)$, and thus
  $\Ess\subobj{}{(A_1, A_2)}{(B_1, B_2)}$ is computed in $\CC_1 \times \CC_2$.)
\end{LEM}
\begin{proof}
  Take any $\Lambda_1 \in \Ess\subobj{}{A_1}{B_1}$ and $\Lambda_2 \in \Ess\subobj{}{A_2}{B_2}$.
  Since $\Lambda_i$ is essential, $i \in \{1, 2\}$, there is a $C_i \in \Ob(\CC_i)$ such that $B_i \to C_i$ and
  for every partition $\Pi_i \in \Part\subobj{}{A_i}{C_i}$ there is
  a $w_i \in \hom(B_i, C_i)$ such that $\Lambda_i \succeq \ell_{w_i}^{-1}(\Pi_i)$.
  To show that $\Lambda_1 \otimes \Lambda_2$ is essential, consider $(C_1, C_2) \in \Ob(\CC_1 \times \CC_2)$ and let
  $\Sigma \in \Part\subobj{}{(A_1, A_2)}{(B_1, B_2)}$ be arbitrary. Let $\chi = \chi_\Sigma : \subobj{}{(A_1, A_2)}{(C_1, C_2)} \to k$
  be a coloring that corresponds to $\Sigma$ so that $\Pi(\chi) = \Sigma$.
  It is easy to show that
  $$\textstyle
    \subobj{}{(A_1, A_2)}{(C_1, C_2)} = \subobj{}{A_1}{C_1} \times \subobj{}{A_2}{C_2}.
  $$
  Therefore, $\chi$ is a coloring $\chi : \subobj{}{A_1}{C_1} \times \subobj{}{A_2}{C_2} \to k$, so it gives rise to
  two colorings:
  $$\textstyle
    \chi_1 : \subobj{}{A_1}{C_1} \to k^{\subobj{}{A_2}{C_2}} \text{\quad and\quad}
    \chi_2 : \subobj{}{A_2}{C_2} \to k^{\subobj{}{A_1}{C_1}}
  $$
  given by
  \begin{align*}
    \chi_1(f / \Boxed{\sim_{A_1}}) = \psi^{(1)}_{f / \Boxed{\sim_{A_1}}} &\text{ where }
      \psi^{(1)}_{f / \Boxed{\sim_{A_1}}}({g / \Boxed{\sim_{A_2}}}) = \chi({f / \Boxed{\sim_{A_1}}}, {g / \Boxed{\sim_{A_2}}}), \text{ and}
    \\
    \chi_2(g / \Boxed{\sim_{A_2}}) = \psi^{(2)}_{g / \Boxed{\sim_{A_2}}} &\text{ where }
      \psi^{(2)}_{g / \Boxed{\sim_{A_2}}}({f / \Boxed{\sim_{A_1}}}) = \chi({f / \Boxed{\sim_{A_1}}}, {g / \Boxed{\sim_{A_2}}}).
  \end{align*}
  Since $\Lambda_i \in \Ess\subobj{}{A_i}{B_i}$, $i \in \{1, 2\}$, there is a $w_i \in \hom_{\CC_i}(B_i, C_i)$ such that
  \begin{equation}\label{rament.eq.lambda-a}
    \Lambda_i \succeq \ell_{w_i}^{-1}(\Pi(\chi_i)), \qquad i \in \{1, 2\}.
  \end{equation}
  Let us show that $\Lambda_1 \otimes \Lambda_2 \succeq \ell_{(w_1, w_2)}^{-1}(\Pi(\chi)) = \ell_{(w_1, w_2)}^{-1}(\Sigma)$.
  Assume that $(f_1 / \Boxed{\sim_{A_1}}, f_2 / \Boxed{\sim_{A_2}}) \equiv_{\Lambda_1 \otimes \Lambda_2}
  (g_1 / \Boxed{\sim_{A_1}}, g_2 / \Boxed{\sim_{A_2}})$. Then for each $i \in \{1, 2\}$ we have that
  $f_i / \Boxed{\sim_{A_i}} \equiv_{\Lambda_i} g_i / \Boxed{\sim_{A_i}}$, so Lemma~\ref{rament.lem.basic-prop}~$(d)$ applied to
  \eqref{rament.eq.lambda-a} yields, for each $i \in \{1, 2\}$,
  $$
    \chi_i(w_i \cdot f_i / \Boxed{\sim_{A_i}}) = \chi_i(w_i \cdot g_i / \Boxed{\sim_{A_i}}).
  $$
  Therefore,
  $
    \psi^{(i)}_{w_i \cdot f_i / \Boxed{\sim_{A_i}}} = \psi^{(i)}_{w_i \cdot g_i / \Boxed{\sim_{A_i}}},
  $
  whence
  $$
    \psi^{(1)}_{w_1 \cdot f_1 / \Boxed{\sim_{A_1}}}(w_2 \cdot f_2 / \Boxed{\sim_{A_2}}) =
    \psi^{(1)}_{w_1 \cdot g_1 / \Boxed{\sim_{A_1}}}(w_2 \cdot f_2 / \Boxed{\sim_{A_2}}),
  $$
  and
  $$
    \psi^{(2)}_{w_2 \cdot f_2 / \Boxed{\sim_{A_2}}}(w_1 \cdot g_1 / \Boxed{\sim_{A_1}}) =
    \psi^{(2)}_{w_2 \cdot g_2 / \Boxed{\sim_{A_2}}}(w_1 \cdot g_1 / \Boxed{\sim_{A_1}}).
  $$
  Recalling the definition of $\psi^{(1)}$ and $\psi^{(2)}$ in terms of $\chi$ we end up with
  $$
    \chi(w_1 \cdot f_1 / \Boxed{\sim_{A_1}}, w_2 \cdot f_2 / \Boxed{\sim_{A_2}}) =
    \chi(w_1 \cdot g_1 / \Boxed{\sim_{A_1}}, w_2 \cdot f_2 / \Boxed{\sim_{A_2}}),
  $$
  and
  $$
    \chi(w_1 \cdot g_1 / \Boxed{\sim_{A_1}}, w_2 \cdot f_2 / \Boxed{\sim_{A_2}}) =
    \chi(w_1 \cdot g_1 / \Boxed{\sim_{A_1}}, w_2 \cdot g_2 / \Boxed{\sim_{A_2}}),
  $$
  whence
  $$
    \chi(w_1 \cdot f_1 / \Boxed{\sim_{A_1}}, w_2 \cdot f_2 / \Boxed{\sim_{A_2}}) =
    \chi(w_1 \cdot g_1 / \Boxed{\sim_{A_1}}, w_2 \cdot g_2 / \Boxed{\sim_{A_2}}).
  $$
  This concludes the proof by Lemma~\ref{rament.lem.basic-prop}~$(d)$.
\end{proof}

\begin{THM}\label{rament.thm.srd-subadd}
  Let $\CC_1$ and $\CC_2$ be small categories whose morphisms are mono and homsets are finite.
  Let $H$ be an entropy on partitions and let
  $\tilde r_{\CC_1}$, $\tilde r_{\CC_2}$ and $\tilde r_{\CC_1 \times \CC_2}$
  be the Ramsey entropies based on $H$. Then for all $X_1 \in \Ob(\CC_1)$ and $X_2 \in \Ob(\CC_2)$:
  $$
    \tilde r_{\CC_1 \times \CC_2}(X_1, X_2) \le \tilde r_{\CC_1}(X_1) + \tilde r_{\CC_2}(X_2).
  $$
\end{THM}
\begin{proof}
  Let us first show that for arbitrary $B_1 \in \Ob(\CC_1)$, $B_2 \in \Ob(\CC_2)$ satisfying $A_1 \to B_1$
  and $A_2 \to B_2$:
  $$
    \min\nolimits_{\Lambda \in \Ess\subobj{}{(A_1, A_2)}{(B_1, B_2)}} H(\Lambda)
    \le \min\nolimits_{\Lambda_1 \in \Ess\subobj{}{A_1}{B_1}} H(\Lambda_1)
    + \min\nolimits_{\Lambda_2 \in \Ess\subobj{}{A_2}{B_2}} H(\Lambda_2).
  $$
  Choose
  $\Lambda_1^* \in \Ess\subobj{}{A_1}{B_1}$ and $\Lambda_2^* \in \Ess\subobj{}{A_2}{B_2}$ so that
  \begin{align*}
    H(\Lambda_1^*) &= \textstyle\min\nolimits_{\Lambda_1 \in \Ess\subobj{}{A_1}{B_1}} H(\Lambda_1),\\
    H(\Lambda_2^*) &= \textstyle\min\nolimits_{\Lambda_2 \in \Ess\subobj{}{A_2}{B_2}} H(\Lambda_2).
  \end{align*}
  We know from Lemma~\ref{rament.lem.L*L-in-Ess} that $\Lambda_1^* \otimes \Lambda_2^* \in \Ess\subobj{}{(A_1, A_2)}{(B_1, B_2)}$, so
  \begin{align*}
    \min\nolimits_{\Lambda \in \Ess\subobj{}{(A_1, A_2)}{(B_1, B_2)}} H(\Lambda) & \le H(\Lambda_1^* \otimes \Lambda_2^*) = H(\Lambda_1^*) + H(\Lambda_2^*)\\
                 &= \min\nolimits_{\Lambda_1 \in \Ess\subobj{}{A_1}{B_1}} H(\Lambda_1) + \min\nolimits_{\Lambda_2 \in \Ess\subobj{}{A_2}{B_2}} H(\Lambda_2).
  \end{align*}
  By majorizing the right-hand side of the above expression we get
  \begin{align*}
    \min\nolimits_{\Lambda_1 \in \Ess\subobj{}{A_1}{B_1}} H(\Lambda_1) &\le \sup\nolimits_{B_1 \in \upset{}A_1} \min\nolimits_{\Lambda_1 \in \Ess\subobj{}{A_1}{B_1}} H(\Lambda_1) = \phi(A_1) \quad \text{and}\\
    \min\nolimits_{\Lambda_2 \in \Ess\subobj{}{A_2}{B_2}} H(\Lambda_2) &\le \sup\nolimits_{B_2 \in \upset{}A_2} \min\nolimits_{\Lambda_2 \in \Ess\subobj{}{A_2}{B_2}} H(\Lambda_2) = \phi(A_2),
  \end{align*}
  whence
  $$
    \min\nolimits_{\Lambda \in \Ess\subobj{}{(A_1, A_2)}{(B_1, B_2)}} H(\Lambda) \le \phi(A_1) + \phi(A_2).
  $$
  Then, by majorizing the left-hand side, we arrive at
  $$
    \phi(A_1, A_2) = \sup\nolimits_{(B_1, B_2) \in \upset{} (A_1, A_2)} \min\nolimits_{\Lambda \in \Ess\subobj{}{(A_1, A_2)}{(B_1, B_2)}} H(\Lambda) \le \phi(A_1) + \phi(A_2).
  $$
  Finally,
  \begin{align*}
    \tilde r(X_1, X_2)
    &= \inf\nolimits_{(A_1, A_2) \in \upset{}(X_1, X_2)} \phi(A_1, A_2)\\
    &\le \inf\nolimits_{(A_1, A_2) \in \upset{}(X_1, X_2)} (\phi(A_1) + \phi(A_2))\\
    &= \inf\nolimits_{(A_1, A_2) \in \upset{}(X_1, X_2)} \phi(A_1) + \inf\nolimits_{(A_1, A_2) \in \upset{}(X_1, X_2)} \phi(A_2)\\
    &= \inf\nolimits_{A_1 \in \upset{}X_1} \phi(A_1) + \inf\nolimits_{A_2 \in \upset{}X_2} \phi(A_2) = \tilde r(X_1) + \tilde r(X_2).
  \end{align*}
  This completes the proof.
\end{proof}

\section{The Ramsey-Boltzmann entropy}
\label{rament.sec.r-b}

In this section we focus on the properties of the Ramsey-Boltzmann entropy as the maximal Ramsey entropy.
Our main result, the additivity of the Ramsey-Boltzmann entropy, is the consequence of the fact that
the behavior of the maximal entropy is governed by the properties of structural Ramsey degrees which
act as the corresponding diversity measure. 
We conclude the section with the discussion of the behavior of the Ramsey-Boltzmann entropy under
forgetful functors.

Let us start by showing that the Ramsey-Boltzmann entropy is the maximal Ramsey entropy on a category.
We start with a useful technical result.

\begin{LEM}\label{rament.lem.phi=log.srd}
  Let $\CC$ be a small category whose morphisms are mono and homsets are finite.
  Then $\phi(A) = \log \tilde t(A)$ for all $A \in \Ob(\CC)$, where we assume $\log \infty = \infty$.
\end{LEM}
\begin{proof}
  Let $H^\Bol$ denote the Boltzmann entropy on partitions.
  Take any $A \in \Ob(\CC)$. We have already seen in Lemma~\ref{rament.lem.phi<=log.srd}
  that $\phi(A) \le \log \tilde t(A)$ for any entropy on partitions, so let us show that
  $\phi(A) \ge \log \tilde t(A)$ in case $\phi$ is based on the Boltzmann entropy.
  If $\phi(A) = \infty$ the statement is trivial. Assume, therefore, that
  $\phi(A) = r \in \RR$. Then for every $B \in \upset{}A$
  there is a $\Pi_B \in \Ess\subobj{}AB$ such that $H^\Bol(\Pi_B) \le r$, or equivalently,
  $|\Pi_B| \le \lfloor 2^r \rfloor$. Lemma~\ref{rament.lem.ess-col}~$(a)$ now yields that $\tilde t(A) \le \lfloor 2^r \rfloor$ whence
  $\log \tilde t(A) \le r$.
\end{proof}

\begin{THM}
  Let $\CC$ be a small category whose morphisms are mono and homsets are finite,
  let $\tilde r'$ be a Ramsey entropy and $\tilde r$ the Ramsey-Boltzmann entropy.
  Then $\tilde r'(X) \le \tilde r(X)$ for all $X \in \Ob(\CC)$.
\end{THM}
\begin{proof}
  Let $\phi'$ and $\phi$ be the functions that correspond to $\tilde r'$ and $\tilde r$, respectively, so that
  $\tilde r'(X) = \inf_{A \in \upset{}X} \phi'(A)$ and $\tilde r(X) = \inf_{A \in \upset{}X} \phi(A)$.
  By Lemmas~\ref{rament.lem.phi<=log.srd} and~\ref{rament.lem.phi=log.srd} we have that
  $\phi'(A) \le \log \tilde t(A) = \phi(A)$ for all $A \in \Ob(\CC)$ whence the claim of the theorem follows immediately.
\end{proof}

The following two theorems are improved versions of the two main  results of Section~\ref{rament.sec.rament}.

\begin{THM}\label{rament.thm.R=log-t}
  Let $\CC$ be a small category with amalgamation whose morphisms are mono and homsets are finite.

  $(a)$ $\CC$ has finite structural Ramsey degrees if and only if $\CC$ admits the Ramsey-Boltzmann entropy.
  
  $(b)$ Assume that $\CC$ admits the Ramsey-Boltzmann entropy $\tilde r$.
  Then $X \in \Ob(\CC)$ is a subramsey object in $\CC$ if and only if $\tilde r(X) = 0$.
\end{THM}
\begin{proof}
  $(a)$ Direction $(\Rightarrow)$ was proved in Theorem \ref{rament.thm.R<=log-t}~$(c)$, so let us prove
  direction $(\Leftarrow)$. Take any $X \in \Ob(\CC)$ and let $\tilde r$ be the Ramsey-Boltzmann entropy.
  By the assumption, $\tilde r(X) < \infty$ for
  every $X \in \Ob(\CC)$. Let $X \in \Ob(\CC)$ be arbitrary and let us show  that $\tilde t(X) < \infty$.
  We know that
  $$
    \tilde r(X) = \inf\nolimits_{A \in \upset{} X} \phi(A) < \infty.
  $$
  Since $\phi(A) = \log \tilde t(A)$ (Lemma~\ref{rament.lem.phi=log.srd}) and $\tilde t(A) \in \NN \union \{\infty\}$, it follows that
  $$
    \{\phi(A) : A \in \upset{}X\} \subseteq \{ \log n : n \in \NN \} \union \{\infty\}.
  $$
  Therefore, there exists an $A_0 \in \upset{}X$ such that $\phi(A_0) = \tilde r(X) < \infty$, or, equivalently,
  $\tilde t(A_0) = 2^{\tilde r(X)} < \infty$. Since $X \to A_0$ by the choice of $A_0$, Lemma~\ref{rament.lem.t-monotonous}
  ensures that $t(X) \le t(A_0)$, whence, with the help of Proposition~\ref{rdbas.prop.sml}:
  $$
    \tilde t(X) \le \frac{|\Aut(A_0)|}{|\Aut(X)|} \tilde t(A_0) < \infty.
  $$
  (Recall that homsets, and in particular automorphisms groups, in $\CC$ are finite).
  
  $(b)$ Again, direction $(\Rightarrow)$ was proved in Theorem \ref{rament.thm.R<=log-t}~$(d)$, so let us prove
  direction $(\Leftarrow)$. Assume that $\tilde r(X) = 0$ for some $X \in \Ob(\CC)$. Then, using the results proved in $(a)$
  of this statement,
  $$
    \tilde r(X) = \inf\nolimits_{A \in \upset{} X} \phi(A) = \min\nolimits_{A \in \upset{} X} \phi(A) = 0.
  $$
  Therefore, there is an $A \in \upset{}X$ with $\phi(A) = 0$. But, as we have seen, $\phi(A) = \log \tilde t(A)$, whence $\tilde t(A) = 1$.
  So, $X$ is a subramsey object.
\end{proof}

\begin{THM}\label{rament.thm.additivity}
  Let $\CC_1$ and $\CC_2$ be small categories with amalgamation whose morphisms are mono and homsets are finite,
  and let $\tilde r_{\CC_1}$ and $\tilde r_{\CC_2}$ be the Ramsey-Boltzmann entropies. Then
  for all $X_1 \in \Ob(\CC_1)$ and $X_2 \in \Ob(\CC_2)$:
  $$
    \tilde r_{\CC_1 \times \CC_2}(X_1, X_2) = \tilde r_{\CC_1}(X_1) + \tilde r_{\CC_2}(X_2).
  $$
\end{THM}
\begin{proof}
  Recall that $\phi(A) = \log \tilde t(A)$ (Lemma~\ref{rament.lem.phi=log.srd}).
  So, it follows from Theorem~\ref{rpppg.thm.srd-mult} that
  $$
    \phi(A_1, A_2) = \phi(A_1) + \phi(A_2),
  $$
  for all $A_1 \in \Ob(\CC_1)$ and $A_2 \in \Ob(\CC_2)$. (Note that $(A_1, A_2)$ is then an object in $\CC_1 \times \CC_2$.) 
  Finally,
  \begin{align*}
    \tilde r(X_1, X_2)
    &= \inf\nolimits_{(A_1, A_2) \in \upset{}(X_1, X_2)} \phi(A_1, A_2)\\
    &= \inf\nolimits_{(A_1, A_2) \in \upset{}(X_1, X_2)} (\phi(A_1) + \phi(A_2))\\
    &= \inf\nolimits_{(A_1, A_2) \in \upset{}(X_1, X_2)} \phi(A_1) + \inf\nolimits_{(A_1, A_2) \in \upset{}(X_1, X_2)} \phi(A_2)\\
    &= \inf\nolimits_{A_1 \in \upset{}X_1} \phi(A_1) + \inf\nolimits_{A_2 \in \upset{}X_2} \phi(A_2) = \tilde r(X_1) + \tilde r(X_2).
  \end{align*}
  This completes the proof.
\end{proof}

\section{Conclusion}

Let us summarize the properties of the Ramsey-Boltzmann entropy by providing a setting where
all the results can be presented in a single category. Let $\CC$ be a small category.
The \emph{state space associated to $\CC$} is the coproduct $\SS = \coprod_{n \in \NN} \CC^n$ computed in $\Cat$, the
category of all small categories.
Explicitly, the objects of $\SS$ are tuples $A = (A_1, \ldots, A_n)$ of objects $A_1, \ldots, A_n \in \Ob(\CC)$
of all possible finite lengths $n \in \NN$, while morphisms exist only between tuples of the same length, and then
$f$ is a morphism $(A_1, \ldots, A_n) \to (B_1, \ldots, B_n)$ if and only if $f = (f_1, \ldots, f_n)$
where $f_i \in \hom_\CC(A_i, B_i)$, $1 \le i \le n$.
Note that $\SS$ comes with a bifunctor $\star : \SS \times \SS \to \SS$ defined to be the concatenation of tuples
both on objects:
$$
  (A_1, \ldots, A_n) \star (B_1, \ldots, B_m) = (A_1, \ldots, A_n, B_1, \ldots, B_m)
$$
and on morphisms:
$$
  (f_1, \ldots, f_n) \star (g_1, \ldots, g_m) = (f_1, \ldots, f_n, g_1, \ldots, g_m).
$$
In particular, $A \star B = (A, B) \in \Ob(\CC^2)$ for $A, B \in \Ob(\CC)$.

\begin{COR}
  Let $\CC$ be a small category with amalgamation whose morphisms are mono and homsets are finite,
  and let $\SS$ be the state space associated to $\CC$. Then
  $\SS$ admits the Ramsey-Boltzmann entropy if and only if $\CC$ has structural Ramsey degrees.
\end{COR}
\begin{proof}
  Note, first, that $\SS$ has amalgamation, that morphisms in $\SS$ are mono and that
  homsets in $\SS$ are finite.
  In view of Theorem~\ref{rament.thm.R=log-t}~$(a)$ it is clear that it suffices to show that
  $\SS$ has structural Ramsey degrees if and only if $\CC$ does. But this follows straightforwardly
  from the fact that $\SS$ is a coproduct of all finite powers of~$\CC$.
  Namely, if $(A_1, \ldots, A_n) \in \Ob(\SS)$ then
  $\tilde t_\SS(A_1, \ldots, A_n) = \tilde t_{\CC^n}(A_1, \ldots, A_n)$
  because $\SS$ is the coproduct of finite powers of $\CC$ and there are no morphisms between objects of
  $\CC^n$ and $\CC^m$ for $n \ne m$. On the other hand,
  $\tilde t_{\CC^n}(A_1, \ldots, A_n) = \prod_{i=1}^n \tilde t_\CC(A_i)$ by Theorem~\ref{rpppg.thm.srd-mult},
  and this holds even if some of the degrees are~$\infty$.
\end{proof}

\begin{COR}
  Let $\CC$ be a small category with amalgamation whose morphisms are mono and homsets are finite
  and assume that $\CC$ has structural Ramsey degrees.
  Let $\SS$ be the state space associated to $\CC$ and let $\tilde r_\SS$ be the Ramsey-Boltzmann entropy.
  Then for all $X, Y \in \Ob(\SS)$:

  $(a)$ $X \toCC\SS Y$ implies $\tilde r_\SS(X) \le \tilde r_\SS(Y)$;

  $(b)$ $\tilde r_\SS(X) \le \log \tilde t_\SS(X)$;

  $(c)$ $\tilde r_\SS(X) = 0$ if and only if $X$ is a subramsey object in $\SS$;

  $(d)$ $\tilde r_\SS(X \star Y) = \tilde r_\SS(X) + \tilde r_\SS(Y)$.
\end{COR}
\begin{proof}
  Note, first, that $\SS$ has amalgamation, that morphisms in $\SS$ are mono and that
  homsets in $\SS$ are finite.

  $(a)$ is Theorem~\ref{rament.thm.R<=log-t}~$(a)$; $(b)$ is Theorem~\ref{rament.thm.R<=log-t}~$(b)$;
  $(c)$ is Theorem~\ref{rament.thm.R=log-t}~$(b)$. Let us show~$(d)$.
  Recall again that $\SS$ is the coproduct of all finite powers of~$\CC$, so
  for every $X = (A_1, \ldots, A_n) \in \Ob(\SS)$
  we have that $\tilde r_\SS(X) = \tilde r_{\CC^n}(A_1, \ldots, A_n)$.
  Now for $X = (A_1, \ldots, A_n) \in \Ob(\SS)$ and
  $Y = (B_1, \ldots, B_m) \in \Ob(\SS)$ Theorem~\ref{rament.thm.additivity}
  gives us that
  \begin{align*}
    \tilde r_\SS(X \star Y)
    &= \tilde r_{\CC^{n+m}}(A_1, \ldots, A_n, B_1, \ldots, B_m)\\
    &= \tilde r_{\CC^n}(A_1, \ldots, A_n) + \tilde r_{\CC^m}(B_1, \ldots, B_m)\\
    &= \tilde r_\SS(X) + \tilde r_\SS(Y).
  \end{align*}
  This completes the proof.
\end{proof}

We conclude the paper with a discussion on the behavior of the Ramsey-Boltzmann
entropy under forgetful functors. Since such a functor takes an object with more structure to an object
with less structure, it makes sense to assume that the entropy should increase along the way.

Let $U : \CC \to \DD$ be a forgetful functor. For $D \in \Ob(\DD)$ let
$
  U^{-1}(D) = \{C \in \Ob(\CC) : U(C) = D \}
$. Note that this is not necessarily a set. However, we say that $U : \CC \to \DD$ is
\emph{finitary} if $U^{-1}(D)$ is a finite set for all $D \in \Ob(\DD)$.
A forgetful functor $U : \CC \to \DD$ is \emph{reasonable} (cf.~\cite{KPT}) if
for every $e \in \hom_\DD(A, B)$ and every $C \in U^{-1}(A)$ there is a $D \in U^{-1}(B)$ and
a morphism $f \in \hom_\CC(C, D)$ such that $U(f) = e$:
\begin{center}
    \begin{tikzcd}
      C \arrow[r, "f"] \arrow[d, dashed, mapsto, "U"'] & D \arrow[d, dashed, mapsto, "U"] \\
      A \arrow[r, "e"'] & B
    \end{tikzcd}
\end{center}
A forgetful functor $U : \CC \to \DD$ has \emph{unique restrictions} \cite{masul-kpt} if
for every $D \in \Ob(\CC)$ and every $e \in \hom_\DD(A, U(D))$ there is a \emph{unique} $C \in U^{-1}(A)$
and a morphism $f \in \hom_\CC(C, D)$ such that $U(f) = e$:
\begin{center}
    \begin{tikzcd}
      C \arrow[r, "f"] \arrow[d, dashed, mapsto, "U"'] & D \arrow[d, dashed, mapsto, "U"] \\
      A \arrow[r, "e"'] & U(D)
    \end{tikzcd}
\end{center}
Following \cite{vanthe-more} we say that $U : \CC \to \DD$ \emph{has the expansion property}
if for every $A \in \Ob(\DD)$ there exists a $B \in \Ob(\DD)$ such that $C \toCC\CC D$ for all
$C \in U^{-1}(A)$ and all $D \in U^{-1}(B)$.

\begin{THM}\label{rament.thm.small2-obj}
  Let $\CC$ and $\DD$ be small categories with amalgamation whose homsets are finite and morphisms are mono. Assume also that
  $\DD$ is a directed category and that the two categories admit Ramsey-Boltzmann entropies
  $\tilde r_\CC$ and $\tilde r_\DD$, respectively.
  Assume that $U : \CC \to \DD$ is a forgetful functor which is reasonable, finitary, with unique restrictions and
  with the expansion property. Then $\tilde r_\DD(U(X)) \ge \tilde r_\CC(X)$ for all $X \in \Ob(\CC)$.
\end{THM}
\begin{proof}
  Take any $A \in \Ob(\CC)$ and let $D = U(A)$. Then $U^{-1}(D)$ is finite.
  Moreover, $\tilde t_{\CC}(B) < \infty$ for all $B \in U^{-1}(D)$ because $\CC$ admits the Ramsey-Boltzmann
  entropy (Theorem~\ref{rament.thm.R=log-t}). Let $B_1$, \ldots, $B_n$ be representatives of isomorphism classes of objects
  in $U^{-1}(D)$. Without loss of generality we can assume that $A \cong B_1$. Then by \cite[Corollary 6.4 $(b)$]{masul-kpt},
  $$
    \tilde t_\DD(U(A)) = \tilde t_{\DD}(D) = \sum_{i=1}^n \tilde t_{\CC}(B_i) \ge \tilde t_{\CC}(B_1) = \tilde t_{\CC}(A).
  $$
  Having in mind that $\phi(A) = \log \tilde t(A)$ (Lemma~\ref{rament.lem.phi=log.srd}), we get that $\phi_\CC(A) \le \phi_\DD(U(A))$.

  Now, let $X \in \Ob(\CC)$ be arbitrary.
  Since $U$ is reasonable for every $D \in \upset{\DD} U(X)$ there is a $B \in \upset{\CC} X$ with $U(B) = D$.
  Therefore, for every $D \in \upset{\DD} U(X)$ there is a $B \in \upset{\CC} X$ with $\phi_\CC(B) \le \phi_\DD(U(B))
  = \phi_\DD(D)$. This immediately implies
  $$
    \tilde r_\CC(X) = \inf\nolimits_{B \in \upset{\CC}X} \phi_\CC(B) \le
    \inf\nolimits_{D \in \upset{\DD}U(X)} \phi_\DD(D) = \tilde r_\DD(U(X)),
  $$
  and the proof is complete.
\end{proof}

\section{Acknowledgement}

This research was supported by the Science Fund of the Republic of Serbia, Grant No.~7750027,
\textit{Set-theoretic, model-theoretic and Ramsey-theoretic phenomena in mathematical structures: similarity and diversity -- SMART}.


\end{document}